\documentclass[12pt]{amsart}
\usepackage[margin=1.25in]{geometry}
\usepackage{latexsym}
\usepackage{amsmath}
\usepackage{amssymb}
\usepackage{amsthm}
\usepackage{stmaryrd}
\usepackage{enumitem}
\usepackage[all]{xy}
\usepackage{mathtools}
\usepackage{xurl}
\usepackage[pagebackref]{hyperref}

\makeatletter
  
  \@addtoreset{equation}{thm}
  \makeatother

\title{Noetherian local rings without perfectoid towers}

\author{Tatsuki Yamaguchi}
\address{Department of Mathematics, School of Science, Institute of Science Tokyo}
\email{yamaguchi.t.bp@m.titech.ac.jp}

\subjclass[2020]{}

\def\ge{\geqslant}
\def\le{\leqslant}
\def\phi{\varphi}
\def\epsilon{\varepsilon}
\def\tilde{\widetilde}
\def\bar{\overline}
\def\mapsto{\longmapsto}

\def\Ker{\operatorname{ker}}

\def\m{{\mathfrak m}}
\def\n{{\mathfrak n}}

\newcommand{\F}{\mathbb{F}}
\newcommand{\N}{\mathbb{N}}
\newcommand{\Z}{\mathbb{Z}}

\newcommand{\gr}{\operatorname{gr}}

\theoremstyle{plain}
\newtheorem{thm}{Theorem}[section]
\newtheorem{cor}[thm]{Corollary}
\newtheorem{prop}[thm]{Proposition}
\newtheorem{lem}[thm]{Lemma}
\theoremstyle{definition}
\newtheorem{defn}[thm]{Definition}
\newtheorem{eg}[thm]{Example}
\theoremstyle{remark}
\newtheorem{rem}[thm]{Remark}
\newtheorem{setting}[thm]{Setting}
\newtheorem{notation}[thm]{Notation}
\newtheorem*{cl}{Claim}
\newtheorem*{clproof}{Proof of Claim}

\newtheorem*{acknowledgement}{Acknowledgments}

\begin{document}

\tolerance = 9999

\begin{abstract}
Let $k$ be a perfect field of characteristic $p>0$. We prove that if a $p$-torsionfree ring $R$, defined by homogeneous equations over $W(k)$, admits a perfectoid tower arising from $(R,(p))$, then its special fiber is $F$-liftable. We also give explicit examples admitting no such tower, including examples with $F$-liftable special fibers.
\end{abstract}

\maketitle
\markboth{}{}

\tableofcontents

\section{Introduction}
	The notion of perfectoid towers was introduced by Ishiro--Nakazato--Shimomoto \cite{INS25} to apply perfectoid methods to the study of Noetherian rings of mixed characteristic. A perfectoid tower $(\{R_i\}_{i\ge 0},\{t_i\}_{i\ge 0})$ consists of a sequence of rings $R_i$ and transition maps $t_i:R_i\to R_{i+1}$ satisfying a certain set of axioms. They proved that any complete local log-regular ring $R$ with perfect residue field of characteristic $p>0$ admits a perfectoid tower arising from $(R,(p))$.

 Ishizuka \cite{Ishizuka26} showed that certain prisms induce perfectoid towers. Hayashi \cite{Hayashi26} studied the ring-theoretic properties of Noetherian rings $R_i$ appearing in perfectoid towers.

Ishiro--Shimomoto \cite{Ishiro-Shimomoto26} developed further constructions using $\delta$-rings and highlighted the existence problem for perfectoid towers over complete Noetherian local domains.
In this paper, we study towers arising from $(R,(p))$ and derive obstructions from their first transition map.

Our first obstruction comes from Frobenius liftability of the special
fiber. Frobenius liftability modulo $p^2$, also called $F$-liftability, has been extensively studied in algebraic geometry. Local $F$-liftability has been studied by Zdanowicz \cite{Zdanowicz18} and Kawakami--Takamatsu \cite{Kawakami-Takamatsu26,Kawakami-Takamatsu26a}; see \cite{BTLM97,AWZ21,AWZ23} for global results.

The notion of a preperfectoid tower generalizes that of a perfectoid tower; see \cite{Hayashi26}.
For rings defined by homogeneous equations over $W(k)$,
we prove the following.

\begin{thm}[Theorem \ref{thm F-liftability of the closed fiber}] \label{main thm 1}
	Let $k$ be a perfect field of characteristic $p>0$. Assign positive weights $\deg x_i=d_i$, let $f_1,\dots,f_m\in W(k)\llbracket x_1,\dots,x_n \rrbracket$ be homogeneous polynomials and let  $R:=W(k)\llbracket x_1,\dots,x_n\rrbracket/(f_1,\dots,f_m)$. Suppose that $R$ is $p$-torsionfree.
	If $R$ admits a preperfectoid tower arising from $(R,(p))$, then $R/pR$ is $F$-liftable.
\end{thm}

We next give explicit nonexistence results.

\begin{thm}[Theorems \ref{thm $x^p+y^p+z^p$}, \ref{thm one dimensional} and \ref{thm $p^2+x^p+yz$}]
	Let $k$ be a perfect field of characteristic $p>0$.
	The following rings $R$ do not admit a preperfectoid tower arising from $(R,(p))$.
	\begin{enumerate}
	\item $R=W(k)\llbracket x,y,z\rrbracket/(x^p+y^p+z^p)$ ($p\ge 5$).
	\item $R=W(k)\llbracket x\rrbracket/(x^b+pG)$, where $b\ge 3$ and $G\in (x^2,p)W(k)\llbracket x \rrbracket$.
	\item $R=W(k)\llbracket x,y,z\rrbracket/(p^2+x^p+yz)$.
	\end{enumerate}
\end{thm}

Case (1) disproves the converse of Theorem \ref{main thm 1}.
In case (3), the special fiber is even an $F$-regular rational double point. The one-dimensional argument also yields the following classification.

\begin{eg}[Corollary \ref{cor $p^a+x^b$}]
	Let $k$ be as above and $R:=W(k)\llbracket x \rrbracket/(p^a+x^b)$ for $a,b\ge 1$.
	Then $R$ admits  a perfectoid tower arising from $(R,(p))$ if and only if one of the following conditions holds.
\begin{enumerate}
	\item $a=1$ or $b=1$.
	\item $p\ge 3$ and $a=b=2$.
\end{enumerate}
\end{eg}

The paper is organized as follows.
In Section 2, we review the basic notions used in the paper.
In Section 3, we prove Theorem \ref{main thm 1}.
Sections 4--6 establish explicit nonexistence results.
Some technical details concerning Fermat hypersurfaces of degree $p$ are deferred to Appendix A.
\begin{acknowledgement}
 The author would like to express his gratitude to Kazuki Hayashi, Ryo Ishizuka and Shou Yoshikawa for valuable comments. The author used OpenAI's ChatGPT (6.1 Sol and 6 Astra) to explore examples and improve the manuscript, and takes full responsibility
for its content.
This work was supported by JSPS KAKENHI Grant Number 24KJ1040.
\end{acknowledgement}

\section{Preliminaries}
\begin{notation}
Fix a prime number $p$.
We use the following notation.
Let $A$ be a ring, $a$ be an element of $A$, and $I$ be an ideal of $A$.
\begin{enumerate}
\item Let $A^\times$ denote the set of units of $A$.
\item We say that $A$ is \textit{$a$-torsionfree} if $a$ is a non-zero-divisor on $A$.
\item We use $\bar{a}$ to denote the equivalence class of $a$ in $A/pA$.
\item For an $\N$-graded ring or its completion with respect to the irrelevant ideal, $[a]_i$ denotes the homogeneous component of $a$ of degree $i\in \N$.
\item Let $\mathrm{ord}_{I}(a):=\sup\{n\in \Z_{\ge 0}\mid a\in I^n\}$. Here $\mathrm{ord}_{I}(a)=\infty$ if $a\in I^{n}$ for any $n$. For $x\in A$, $\mathrm{ord}_{x}(a)$ denotes $\mathrm{ord}_{(x)}(a)$.
\item We write $\gr_I A:=\bigoplus_{n\ge0}I^n/I^{n+1}$.
If $r=\operatorname{ord}_I(a)<\infty$, then $\operatorname{in}_I(a)$ denotes the class of $a$ in $I^r/I^{r+1}$; otherwise, we set $\operatorname{in}_I(a)=0$.
\item We write $W(A)$ and $W_2(A)$ for the rings of $p$-typical Witt vectors of infinite length and length $2$, respectively.
\item For a perfect field $k$ of characteristic $p>0$, let $\phi$ denote the Frobenius morphism on $W(k)$. For power series rings $W(k)\llbracket x_1,\dots,x_n\rrbracket$, we use the same notation $\phi$ to denote the Frobenius lift such that $\phi(x_i)=x_i^p$. For $f\in W(k)\llbracket x_1,\dots,x_n\rrbracket$, we define
\[
	\delta(f):=\frac{\phi(f)-f^p}{p}\in W(k)\llbracket x_1,\dots,x_n\rrbracket.
\]
\end{enumerate}
\end{notation}
\subsection{Perfectoid towers}
We recall only the properties of the first step of a preperfectoid tower used below; see \cite{INS25,Hayashi26} for the definitions.

For a ring $R$, a \textit{(pre)perfectoid tower} $(\{R_i\}_{i\ge 0},\{t_i\}_{i\ge0})$ is an inductive system of rings $R_0\xrightarrow{t_0} R_1\xrightarrow{t_1}\dots$ satisfying a certain set of axioms.

Let $p$ be a prime number, and let $(R,\m)$ be a Noetherian local ring with $\operatorname{char} R/\m=p$. Suppose that $R$ is $p$-torsionfree. If $R$ admits a preperfectoid tower arising from $(R,(p))$, then the following setting is satisfied.

\begin{setting} \label{setting perfectoid tower}
	$R_1$ is a $p$-torsionfree ring, and $t_0:R\to R_1$ and $F_0:R_1/pR_1\to R/pR$ are ring homomorphisms satisfying the following conditions:
\begin{enumerate}
	\item The morphism $\bar{t_0}:R/pR\to R_1/pR_1$ induced by $t_0$ is injective.
	\item The composition $R_1/pR_1\xrightarrow{F_0} R/pR \xrightarrow{\bar{t_0}} R_1/pR_1$ is equal to the Frobenius morphism on $R_1/pR_1$.
	\item There exists an element $\pi \in R_1$ such that $pR_1=\pi^p R_1$. Let $u\in R_1^\times$ be a unit such that $\pi^p=pu$.
	\item $F_0$ is surjective, and $\Ker F_0=\pi (R_1/pR_1)$.
\end{enumerate}
\end{setting}
We will use $A$ to denote $R/pR$.

\begin{rem}
	The injectivity of $\bar{t_0}$ will only be used in Section 5. The injectivity of $t_0$ will not be used in this paper.
\end{rem}

We will not assume that $R_1$ is Noetherian, but we can easily circumvent this difficulty thanks to the following lemma.

\begin{lem} \label{lem completion of R1}
	Let $(R,\m)$ be a complete Noetherian local ring with $p\in \m$, let $R_1$ be any ring and let $\pi \in R_1$.
	Suppose that $R_1/\pi R_1\cong R/pR$ and $R_1$ is $\pi$-torsionfree. Let $\m_1$ be the contraction of $\m (R/pR)$ by the morphism $R_1\to R_1/\pi R_1\to R/pR$ and let $S$ be the completion of $R_1$ with respect to $\m_1$. Then $S$ is Noetherian and $\pi$-torsionfree, and $S/\pi \cong R/pR$.
\end{lem}
\begin{proof}
	Let $C$ be the completion of $R_1$ with respect to $\pi R_1$. Since $R_1/\pi R_1$ is Noetherian and $\pi R_1$ is a finitely generated ideal, $C$ is Noetherian by \cite[05GH]{stacks-project}.

To see that $C$ is $\pi$-torsionfree, let $\pi x=0$. For each $n$, choose $x_n\in R_1$ representing $x$ modulo $\pi^n$. Then $\pi x_n\in\pi^{n+1}R_1$, so $x_n\in\pi^nR_1$ and hence $x=0$.

Since $C$ is Noetherian and $S$ is the completion of $C$ with respect to the maximal ideal, $S$ is also a Noetherian local ring and $\pi$-torsionfree.
Since $R$ is complete, $S/\pi\cong R/pR$.
\end{proof}

\subsection{$F$-liftability}
In this subsection, we recall the definition of (local) $F$-liftability.
\begin{defn}[{\cite[Definition 1.1]{Zdanowicz18}}]
	Let $k$ be a perfect field of characteristic $p$. A $k$-algebra $R$ is said to be \textit{$F$-liftable} if there exists a flat $W_2(k)$-algebra $S$ such that the diagram
\[
	\xymatrix{
	W_2(k) \ar[r]^{\phi} \ar[d] & W_2(k) \ar[d]\\
	S \ar[r]_{\psi} & S
	}
\]
commutes, $S\otimes_{W_2(k)}k\cong R$ and the reduction of $\psi$ modulo $p$ is $F_R$,
where $F_R$ denotes the Frobenius morphism on $R$.
\end{defn}

We will use Zdanowicz's criterion for $F$-liftability (see \cite[Lemma 4.5 and Theorem 4.6]{Zdanowicz18}).

\section{Perfectoid towers and $F$-liftability of the special fiber}
\label{section F-liftable}
In this section, we derive a necessary condition for the existence of a preperfectoid tower from its first transition map. For rings defined by homogeneous equations, this condition
implies $F$-liftability of the special fiber.

\begin{notation}
Let $k$ be a perfect field of characteristic $p>0$.
For a power series $f=\sum_{\alpha}c_\alpha x^\alpha\in W(k)\llbracket x_1,\dots,x_n \rrbracket$, we define
\[
f^{\phi^{-1}}:=\sum_{\alpha}\phi^{-1}(c_\alpha)x^\alpha.
\]
\end{notation}

\begin{prop} \label{prop first obstruction}
Let $f_1,\dots,f_m\in W(k)\llbracket x_1,\dots,x_n \rrbracket$ be elements and let  $R:=W(k)\llbracket x_1,\dots,x_n\rrbracket/(f_1,\dots,f_m)$ be the quotient ring.
Suppose that $R$ is $p$-torsionfree and admits a preperfectoid tower arising from $(R,(p))$. Then there exist a unit $v\in R$ and elements $g_i, c_j\in R$ such that
\[
	\delta(f_i) \equiv vg_i^p+\sum_{j=1}^{n} c_j \left(\frac{\partial f_i}{\partial x_j}\right)^p \pmod{pR}
\]
for $i=1,\dots, m$.
\end{prop}
\begin{proof}
	Take $\pi\in R_1$ as in Setting \ref{setting perfectoid tower}.
	For $j=1,\dots,n$, take $X_j\in R_1$ such that $F_0(\bar{X_j})=\bar{x_j}$. Since $\bar{t_0}\circ F_0$ is the Frobenius on $R_1/pR_1$, there exist $a_1, \dots, a_n \in R_1$ such that
\[
	t_0(x_j)=X_j^p+pa_j.
\]
Let $S$ be the $(\pi,X_1,\dots,X_n)$-adic completion of $R_1$. Since $\bar{f_i^{\phi^{-1}}(X_1,\dots,X_n)}$ is in $\Ker F_0:S/pS \to R/pR$ for $1\le i \le m$, there exists $h_i \in S$ such that
$f_i^{\phi^{-1}}(X_1,\dots,X_n)=\pi h_i$. By the equality $\pi^p=pu$,
	\begin{align} \label{first ob equation 1}
		f_i(X_1^p,\dots,X_n^p)=p(uh_i^p+(\delta(f_i))^{\phi^{-1}}(X_1,\dots,X_n)).
	\end{align}
On the other hand, we have
\begin{align} \label{first ob equation 2}
	0=t_0(f_i)\equiv f_i(X_1^p,\dots,X_n^p) +\sum_{j=1}^n pa_j\frac{\partial f_i}{\partial x_j}(X_1^p,\dots,X_n^p) \pmod{p^2 S}.
\end{align}
Substituting \eqref{first ob equation 1} into \eqref{first ob equation 2} and dividing by $p$, we obtain, for $1\le i \le m$,
\[
	(\delta(f_i))^{\phi^{-1}}(X_1,\dots,X_n)\equiv -uh_i^p-\sum_{j=1}^n a_j \frac{\partial f_i}{\partial x_j}(X_1^p,\dots,X_n^p) \pmod{\pi S}.
\]
Applying the isomorphism $S/\pi S\cong A$ induced by $F_0$,
we have
\[
	\bar{\delta(f_i)} = -F_0(\bar{u}) F_0(\bar{h_i})^p-\sum_{j=1}^n F_0(\bar{a_j})\left(\bar{\frac{\partial f_i}{\partial x_j}} \right)^p \in A
\]
for $1\le i \le m$.
\end{proof}

\begin{thm} \label{thm F-liftability of the closed fiber}
	Assign positive weights $\deg x_i=d_i$, let $f_1,\dots,f_m\in W(k)\llbracket x_1,\dots,x_n \rrbracket$ be homogeneous polynomials and let  $R:=W(k)\llbracket x_1,\dots,x_n\rrbracket/(f_1,\dots,f_m)$. Suppose that $R$ is $p$-torsionfree.
	If $R$ admits a preperfectoid tower arising from $(R,(p))$, then $R/pR$ is $F$-liftable.
\end{thm}
\begin{proof}
By Proposition \ref{prop first obstruction},
there exist a unit $v\in R$ and elements $g_i,c_j\in R$ such that
\begin{align} \label{first obstruction 1}
	\delta(f_i) \equiv vg_i^p+\sum_{j=1}^{n} c_j \left(\frac{\partial f_i}{\partial x_j}\right)^p \pmod{pR}
\end{align}
for $i=1,\dots, m$. We choose $g_i\in R$, lifting $F_0(\bar{h_i})$, with the notation as in the proof of Proposition \ref{prop first obstruction}.

Let $e_i:=\deg f_i$ for $1\le i\le m$.
Multiplying by $v^{-1}$ and taking the $pe_i$-th homogeneous part of (\ref{first obstruction 1}), we obtain
\[
	[v^{-1}]_0\delta(f_i)\equiv [g_i]_{e_i}^p+\sum_{j=1}^n[v^{-1}c_j]_{pd_j}\left(\frac{\partial f_i}{\partial x_j}\right)^p \pmod{pR}.
\]
Note that $[\bar{v}^{-1}]_0\in k^\times$ has a $p$-th root $\lambda^{-1}$ since $k$ is perfect.

Let $T:=W(k)\llbracket x_1,\dots,x_n \rrbracket$, choose a lift $G_i\in T$ of $-\lambda[g_i]_{e_i}\in A$, let $\tilde{f_i}:=f_i+pG_i$ for $i=1,\dots,m$, and let $\tilde{R}:=W_2(k)\llbracket x_1,\dots,x_n \rrbracket/(\tilde{f_1},\dots,\tilde{f_m})$. $\tilde{R}$ has a Frobenius lift by \cite[the proof of Lemma 4.5 and the subsequent argument]{Zdanowicz18}.

Therefore, it is enough to show that $\tilde{R}$ is flat over $W_2(k)$.

\begin{cl}
	For $b_i\in T/pT$, suppose that $\sum_{i=1}^m b_i\bar{f_i}=0 \in T/pT.$
	Then $\sum_{i=1}^mb_i\bar{g_i}=0\in R/pR.$
\end{cl}
\begin{clproof}
	Let $\tilde{b_i}\in T$ be lifts of $b_i$. Then $\sum_{i=1}^m \tilde{b_i}f_i\in pT.$
	Since $R$ is $p$-torsionfree, there exist $b_i'\in T$ such that	$\sum_{i=1}^m\tilde{b_i}f_i=p(\sum_{i=1}^mb_i'f_i).$
	Hence, $\sum_{i=1}^m(\tilde{b_i}-pb_i')f_i=0.$
	Let $B_i=\tilde{b_i}-pb_i'\in T$. Then we obtain
	\begin{align*}
		0&=\sum_{i=1}^m (B_i^{\phi^{-1}}f_i^{\phi^{-1}})(X_1,\dots,X_n) \\
		&= \pi \sum_{i=1}^m B_i^{\phi^{-1}}(X_1,\dots,X_n)h_i.
	\end{align*}
	Dividing by $\pi$ and applying $S/\pi S \xrightarrow{F_0} R/pR$, we get $\sum_{i=1}^m b_ig_i=0.$
\end{clproof}

Applying the claim degree by degree, we see that
\[
  \sum_{i=1}^m b_i\bar{f_i}=0 \text{ in } T/pT
  \quad\Longrightarrow\quad
  \sum_{i=1}^m b_i\bar{G_i}=0 \text{ in } A.
\]
Indeed, for a homogeneous relation of degree $l$, taking the degree-$l$ component of $\sum_i b_i g_i=0$ gives $\sum_i b_i[g_i]_{e_i}=0$; the general case follows by decomposing the relation into homogeneous components.

Let $\alpha\in T$ be such that $p\alpha=0 \in \tilde{R}.$
Take $\beta_i\in T$ such that
\[
  p\alpha\equiv\sum_{i=1}^m\beta_i\tilde{f_i}
  \pmod{p^2T}.
\]
Reducing modulo $p$ gives
$\sum_i\bar{\beta_i}\bar{f_i}=0$, so $\sum_i\bar{\beta_i}\bar{G_i}=0$ in $A$.
Passing to $R$, we have
\[
	p(\alpha-\sum_{i=1}^m\beta_iG_i)\in p^2R.
\]
Since $R$ is $p$-torsionfree, we obtain
\[
  \alpha\equiv\sum_{i=1}^m\beta_iG_i
  \equiv0\pmod{pR}.
\]
Thus $\alpha\in(f_1,\dots,f_m,p)T=(\tilde{f_1},\dots,\tilde{f_m},p)T$.
Hence, $\tilde R$ is flat over $W_2(k)$.
\end{proof}

\begin{eg}
	Let $k$ be an algebraically closed field of characteristic
$p>0$, and assign positive weights $\deg x_i=d_i$. Let $f_1,\dots,f_m$ be homogeneous elements in $W(k)[x_1,\dots,x_n]$ and let $R:=W(k)\llbracket x_1,\dots,x_n\rrbracket /(f_1,\dots,f_m)$. Suppose that $R$ is a $p$-torsionfree complete intersection of dimension $d\ge 4$ and that $R/pR$ has an isolated singularity.

If $R$ admits a preperfectoid tower arising from $(R,(p))$, then $R/pR$ is $F$-liftable by Theorem \ref{thm F-liftability of the closed fiber}. By \cite[Theorem A]{Kawakami-Takamatsu26}, $R/pR$ is regular, or $\dim R/pR=3$ and
\[
	R/pR\cong k\llbracket x,y,z,w\rrbracket/(xy-zw).
\]
In the latter case, homogeneity gives a presentation
\[
  R\cong W(k)\llbracket x,y,z,w\rrbracket/(xy-zw+pf)
\]
with $f\in(x,y,z,w)^2W(k)\llbracket x,y,z,w\rrbracket$.
A formal change of coordinates over $W(k)$ eliminates $pf$.
Consequently, $R$ is regular, or $d=4$ and
\[
  R\cong W(k)\llbracket x,y,z,w\rrbracket/(xy-zw).
\]
\end{eg}

\section{Fermat hypersurfaces of degree $p$} \label{section second obstruction}
	The obstruction in Section \ref{section F-liftable} does not rule out perfectoid towers when the special fiber is $F$-liftable. In this section, we derive a further obstruction in the case of $W(k)\llbracket x,y,z\rrbracket/(x^p+y^p+z^p)$.

\begin{thm} \label{thm $x^p+y^p+z^p$}
	Let $k$ be a perfect field of characteristic $p$ and let $R:=W(k)\llbracket x,y,z \rrbracket/(x^p+y^p+z^p)$.
	If $p\ge 5$, then $R$ does not admit a preperfectoid tower arising from $(R,(p))$.
\end{thm}

We first derive a necessary condition on polynomial coefficients for the existence of a preperfectoid tower, and then show that it leads to a contradiction.
The proof of Lemma \ref{lem evaluation g} is deferred to Appendix \ref{appendix homogeneous component}.

\subsection{A constraint on the coefficients} \label{subsection coefficient obstruction}

\begin{setting}
	Let $p\ge 3$ be a prime number, let $k$ be a perfect field of characteristic $p$, and let $R:=W(k)\llbracket x,y,z \rrbracket/(x^p+y^p+z^p)$.
\end{setting}
	Suppose that $R$ admits a preperfectoid tower arising from $(R,(p))$. Fix $X,Y,Z\in R_1$ such that $F_0(\bar{X})=\bar{x}$, $F_0(\bar{Y})=\bar{y}$ and $F_0(\bar{Z})=\bar{z}$ in $A$. Take $a,b,c\in R_1$ such that
\[
	t_0(x)=X^p+pa, \quad t_0(y)=Y^p+pb \text{, and } t_0(z)=Z^p+pc.
\]
	We obtain
\begin{align} \label{second ob expansion eq}
	0=t_0(x^p+y^p+z^p) \equiv X^{p^2}+Y^{p^2}+Z^{p^2}+p^2D \pmod{p^3 R_1},
\end{align}
where
\[
	D:=X^{p(p-1)}a+Y^{p(p-1)}b+Z^{p(p-1)}c.
\]
Since $F_0(\bar{X^p+Y^p+Z^p})=0$, there exists $h\in R_1$ such that
\[
	X^p+Y^p+Z^p=\pi h.
\]
Let $L:=X+Y+Z$ and let $(S,\m_S)$ be the completion of $R_1$ with respect to $(\pi,X,Y,L)R_1$. By Lemma \ref{lem completion of R1}, $S$ is a complete Noetherian local ring and is $\pi$-torsionfree. Moreover,
\[
	S/\pi S\cong R_0/pR_0\cong k\llbracket x,y,l\rrbracket/(l^p),
\]
where $l:=x+y+z$. The $W(k)$-algebra structure on $S$ is induced by $W(k)\to R\xrightarrow{t_0}S$. Under the above identification, its reduction on coefficients is the Frobenius automorphism of $k$. Thus, for $a\in W(k)$, both $F_0$ and $\rho$ below send the class of $t_0(a)$ to $\bar a^p$.

Define the polynomial $\Delta(x,y,z)$ by
\[
	\Delta(x,y,z):=\delta(x+y+z)\in \Z_p[x,y,z].
\]
We use $\Delta$ to denote $\Delta(X,Y,Z)\in S$.
 Let  $C:=A/(l)$ and  let $\rho: S/pS\xrightarrow{F_0} A\to C$ be the quotient morphism.

Using $\pi^p=pu$ and dividing (\ref{second ob expansion eq}) by $p$, we obtain
\begin{align} \label{second ob expansion D}
	uh^p+\Delta(X^p,Y^p,Z^p)+pD \equiv 0 \pmod{p^2R_1}.
\end{align}

Define the polynomial $\Gamma(x,y,z)$ by
\[
	\Gamma(x,y,z):=-\delta(\delta(x+y+z)) \in \Z_p[x,y,z].
\]

Hence,
\begin{align} \label{second ob eq}
	uh^p+\Delta^p+p(-\Gamma(X,Y,Z)+D)\equiv 0 \pmod{p^2S}.
\end{align}
Let
\[
	G:=\frac{uh^p+\Delta^p}{p}\in S.
\]
Dividing (\ref{second ob eq}) by $p$, reducing modulo $p$ and then applying $F_0$,
we obtain
\begin{align} \label{second ob eq 2}
	\bar{\Gamma(x,y,z)}= F_0(\bar{G})+\bar{x}^{p(p-1)}F_0(\bar{a})+\bar{y}^{p(p-1)}F_{0}(\bar{b})+\bar{z}^{p(p-1)}F_0(\bar{c}) \in A.
\end{align}

We introduce the following notation:
\begin{enumerate}
\item Set $g:=\rho(\bar{G}) \in C$.
\item Let $M_i=x^{p^2-1-ip}y^{1+ip}\in k[x,y]$ for $1\le i \le p-2$. We define $\mathbf{v}:=\sum_{i=1}^{p-2}(-1)^iM_i$.
\item Let $\gamma(x,y):=\bar{\Gamma(x,y,-x-y)}\in \F_p[x,y]\subseteq k[x,y]$ and let $\gamma_M(x,y)$ denote the sum of the terms of $\gamma(x,y)$ whose monomial parts are among the $M_i$.
\end{enumerate}
\begin{lem} \label{lem evaluation g}
After suitably changing the lift $Z\in S$ of $z$ and redefining the associated quantities, the $p^2$-th homogeneous component $[g]_{p^2}$ of $g$ belongs to $k[x^p,y^p]$.
\end{lem}
The proof is given in Appendix \ref{appendix homogeneous component}. We henceforth use the lifts supplied by this lemma.

\begin{prop} \label{secondary bridge lemma}
	With notation as above, $\gamma_M(x,y)\in k \mathbf{v}$.
\end{prop}
\begin{proof}

By \eqref{second ob eq 2} and Lemma \ref{lem evaluation g},
only $(x+y)^{p(p-1)}[\rho(\bar c)]_p$ contributes to $\gamma_M$.
In this product, only the $x^{p-1}y$ term of $[\rho(\bar c)]_p$
contributes. Denoting its coefficient by $\lambda'\in k$, we obtain
\[
  \gamma_M
  =\lambda'\sum_{i=1}^{p-2}\binom{p-1}{i}M_i
  =\lambda'\mathbf v.
\]

\end{proof}

\subsection{Polynomial identities and root relations}

For a power series $f(t)$, let $[t^j]f(t)$ denote the coefficient of $t^j$ in $f(t)$.

We write
\[
	\gamma_{M}(x,y)=\sum_{i=1}^{p-2}c_iM_i.
\]
Here $c_i\in\F_p$, since $\Gamma\in\Z_p[x,y,z]$. By Proposition \ref{secondary bridge lemma}, $c_i=(-1)^{i-1}c_1$. Thus the remaining polynomial computations take place over $\F_p$.
As this subsection only involves polynomial computations, all previous notation is discarded except for $p$, $\Gamma$, $\gamma_M$, $M_i$ and $c_i$.
Let
\[
	h(t):=\frac{(1+t)^p-1-t^p}{p}, \ \tilde{G}(t):=\frac{h(t)}{t}
\]
and $G$ be the reduction of $\tilde{G}$ modulo $p$. Then
\begin{align} \label{polynomial G}
	G(t)=\sum_{i=0}^{p-2}\frac{(-1)^i}{i+1}t^i\in \F_p[t].
\end{align}

\begin{lem}
	For $1\le i \le p-2$,
	\[
		c_i\equiv (-1)^i-\frac{1}{p}[t^{1+ip}]h^p \pmod{p},
	\]
	and
	\[
		c_1=-\frac{1}{2}.
	\]
\end{lem}
\begin{proof}
	Considering $\Gamma(1,t,-1-t)$, we obtain
	\[
		c_i \equiv \frac{1}{p^2}\binom{p^2}{1+ip}-\frac{1}{p}[t^{1+ip}]h^p \pmod{p}.
	\]
	By Lucas' theorem,
	\[
		\frac{1}{p^2}\binom{p^2}{1+ip}\equiv \binom{p-1}{i}\equiv (-1)^i \pmod{p}.
	\]
	Since
	\[
		h\equiv t+\frac{p-1}{2}t^2 \pmod{t^3},
	\]
	it follows that
	\[
		h^p\equiv t^p+\frac{p(p-1)}{2}t^{p+1}\pmod{t^{p+2}}.
	\]
	Therefore, $c_1=-1/2$.
\end{proof}

Let $u:F_*(\F_p\llbracket t \rrbracket) \to \F_p\llbracket t\rrbracket$ be the $\F_p\llbracket t \rrbracket$-linear map
such that
	\[
		u(\sum_{i\ge 0} a_it^i):=\sum_{i\ge 0}a_{ip}t^i,
	\]
	where $F_*$ denotes the Frobenius pushforward.

Let
\[
	b_i:=\frac{1}{p}[t^{1+ip}]h^p\in \F_p.
\]
We must have
\[
	b_i=\frac{(-1)^i}{2}
\]
for $1\le i \le p-2$.
We have
\begin{align*}
	b_i &\equiv [t^{ip}]\bigl(t^{p-1}\tilde{G}^p+t^p\tilde{G}^{p-1}(\tilde{G})'\bigr) \\
	 & \equiv [t^{(i-1)p}](G^{p-1}G') \pmod{p}.
\end{align*}

Let $B(t):=\sum_{i=0}^{p-3}b_{i+1}t^i=u(G^{p-1}G')=Gu\left(G'/G\right)$.
Let
\[
T:=\sum_{i=0}^{p-3}\frac{(-1)^{i+1}}{2}t^i=-\frac{1+t^{p-2}}{2(1+t)}.
\]
The coefficient relations imply
\begin{align} \label{B=T}
	B=T.
\end{align}
We derive a contradiction from this identity.

\begin{lem}
	The polynomial $G$ is palindromic, satisfies $G(-1)=0$, and
	\[
  		G+tG'=(1+t)^{p-1}-t^{p-1}.
	\]
\end{lem}
\begin{proof}
	By \eqref{polynomial G}, $G$ is palindromic and $G(-1)=0$. Differentiating $h=t\tilde{G}$, we obtain the last assertion.
\end{proof}

\begin{prop} \label{prop roots of G}
	With notation as above, let $\beta\in \bar{\F_p}$ be a root of $G$.
	 Then the following statements hold.
	\begin{enumerate}
		\item  $B(\beta^p)=\beta^{p-1}G'(\beta^p)$.
		\item If $B=T$, then $\beta^p=-\beta^2$.
	\end{enumerate}
\end{prop}
\begin{proof}
	For a root $\alpha\in \bar{\F}_p$ of $G$, let $m_\alpha$ be the multiplicity of $\alpha$.
	Then
	\[
		\frac{G'}{G}=\sum_{\alpha}\frac{m_\alpha}{t-\alpha}.
	\]
	We extend $u$ to a $\bar{\F_p}$-linear map $\bar{\F_p}\llbracket t \rrbracket \to \bar{\F_p}\llbracket t \rrbracket$.
	Since
	\[
	u\left(\frac{1}{t-\beta}\right)=\frac{\beta^{p-1}}{t-\beta^p},
	\]
	it follows that
	\[
		B=\sum_{\alpha} \frac{m_\alpha\alpha^{p-1}G}{t-\alpha^p}.
	\]
	Since Frobenius preserves multiplicities, evaluation at $\beta^p$ gives $B(\beta^p)=\beta^{p-1}G'(\beta^p)$;
	both sides vanish if $\beta$ is a multiple root.
	To prove (2),
	set $\gamma=\beta^p$. If $\gamma\neq -1$, then
	\[
	G(\gamma)+\gamma G'(\gamma)=\frac{1-\gamma^{p-1}}{1+\gamma},
	\]
	and
	\[
		T(\gamma)=-\frac{1+\gamma^{p-2}}{2(1+\gamma)}.
	\]
	Hence, by $B=T$,
	\begin{align} \label{beta gamma eq}
		2(1-\gamma^{p-1})+\beta(1+\gamma^{p-2})=0.
	\end{align}
	Since $G$ is palindromic, $\beta^{-1}$ satisfies the same equation, i.e.,
	\[
		2(1-\gamma^{-(p-1)})+\beta^{-1}(1+\gamma^{-(p-2)})=0.
	\]
	Hence,
	\[
		-2\beta(1-\gamma^{p-1})+\gamma(1+\gamma^{p-2})=0.
	\]
	Therefore,
	\[
		(\beta^2+\gamma)(1+\gamma^{p-2})=0.
	\]
	If $1+\gamma^{p-2}=0$, then $1-\gamma^{p-1}=0$ by (\ref{beta gamma eq}). Hence, $\gamma=-1$, which is a contradiction. Therefore, $\beta^p=\gamma=-\beta^2$. This equality also holds for $\beta=-1$.
\end{proof}

\begin{proof}[Proof of Theorem \ref{thm $x^p+y^p+z^p$}]
	Let $\beta_1,\dots,\beta_{p-2}\in \bar{\F_p}$ be all the roots of $G$ counted with multiplicity. Since Frobenius permutes $\beta_1,\dots,\beta_{p-2}$, by Proposition \ref{prop roots of G} (2),
\[
	\sum_{i=1}^{p-2}\beta_i=\sum_{i=1}^{p-2}\beta_i^p=-\sum_{i=1}^{p-2}\beta_i^2.
\]

Since
\[
	G=t^{p-2}-\frac{1}{2}t^{p-3}+\frac{1}{3}t^{p-4}+(\text{lower terms}),
\]
we obtain
\[
	\sum_{i=1}^{p-2}\beta_i=\frac{1}{2}, \ \sum_{i=1}^{p-2}\beta_i^2=\left(\frac{1}{2}\right)^2-2\cdot \frac{1}{3}=-\frac{5}{12}.
\]
However, $1/2\neq 5/12 \in \F_p$ for any $p\ge 5$, which is a contradiction.
\end{proof}

\section{One-dimensional hypersurfaces}

We next study one-dimensional hypersurfaces of the form $W(k)\llbracket x\rrbracket/(F)$.
We establish a necessary condition for the existence of perfectoid towers and use it to determine exactly when the ring $W(k)\llbracket x\rrbracket/(p^a+x^b)$ admits a perfectoid tower.

\begin{setting} \label{setting one variable}
	Let $F\in W(k)\llbracket x \rrbracket$ be a nonzero nonunit and $R:=W(k)\llbracket x \rrbracket/(F)$. Suppose that $R$ is $p$-torsionfree. Then there exists $b\in \Z_{> 0}$ such that
\[
	R/pR\cong k[x] /(x^b).
\]
Multiplying $F$ by a unit if necessary, we may assume that
\[
	F=x^b+pG
\]
for some $G\in W(k)\llbracket x \rrbracket$.
Suppose that $R$ admits a preperfectoid tower arising from $(R,(p))$. Take $X\in R_1$ such that $F_0(\bar{X})=\bar{x}\in A$. We also fix $h\in R_1$ such that $X^b=\pi h$. Let $h_0:=F_0(\bar{h})\in A$.
\end{setting}

\begin{lem} \label{first reduction lemma}
	Suppose that $b\ge 3$ and $h_0\in x A$. Then $\mathrm{ord}_x(h_0)=1$.
\end{lem}
\begin{proof}
	Suppose $\mathrm{ord}_x(h_0)\ge 2$. Then $h \in (X^2, \pi)R_1$.
Hence,
\begin{align} \label{one dimensional eq}
	X^b=\pi h\in (\pi X^2, \pi^2)R_1.
\end{align}
We show the following claim.
\begin{cl}
	For any $n\ge 2$, $X^{n(b-1)}\in (\pi^nX,\pi^{n+1})R_1$.
\end{cl}
\begin{clproof}
	We show the claim by induction on $n$.
	If $n=2$, then, by (\ref{one dimensional eq}),
	\[
		X^{2(b-1)}\in X^{b-2}(\pi X^2,\pi^2)R_1=(\pi X^{b}, \pi^2 X^{b-2}) R_1\subseteq (\pi^2 X,\pi^3)R_1.
	\]
	Suppose that $n\ge 3$.
	Then, by the induction hypothesis,
	\begin{align*}
		X^{n(b-1)}\in X^{b-1}(\pi^{n-1}X,\pi^n)R_1\subseteq (\pi^{n-1}X^b,\pi^{n}X^{b-1})R_1\subseteq (\pi^n X,\pi^{n+1})R_1.
	\end{align*}
\end{clproof}

	By the above claim, we have $X^{p(b-1)}\in \pi^{p}R_1=pR_1$. On the other hand, since $\bar{t_0}:R_0/pR_0\to R_1/pR_1$ is injective by an axiom of preperfectoid towers, we see that
\[
	(\bar{X})^{p(b-1)}=\bar{t_0}(x^{b-1})\neq 0 \in R_1/pR_1,
\]
which is a contradiction.
\end{proof}

If $b\ge 3$ and $h_0\in xA$, then by Lemma \ref{first reduction lemma},
there exist $v,\alpha \in R_1$ such that $F_0(\bar{v})\in A^\times$ and
\begin{align} \label{h formula}
	h=Xv+\pi \alpha.
\end{align}
Hence,
\[
	X^b=\pi X v +\pi^2 \alpha.
\]
\begin{lem} \label{second reduction lemma}
	Suppose that $b\ge 3$ and $h_0\in xA$.
	For $N=n(b-1)+r$, where $n\ge 0$ and $1\le r\le b-1$,
	\[
		X^N\in \pi^n X^rv^n+(\pi^nX^{r+1},\pi^{n+1})R_1.
	\]
	In particular, $\mathrm{ord}_{\pi}(X^N)=n$.
\end{lem}
\begin{proof}
	We show the lemma by induction on $n$. The case $n=0$ is immediate. If $n\ge 1$, then
	\[
		X^{N}=X^{N-b}X^b=\pi X^{N-(b-1)}v+\pi^2 X^{N-b}\alpha.
	\]
	By the induction hypothesis,
	\[
		\pi X^{N-(b-1)}v\in \pi^n X^{r} v^n+(\pi^nX^{r+1},\pi^{n+1})R_1.
	\]
	If $n=1$ or $r\ge 2$, then $\pi^2 X^{N-b}\in \pi^{n+1}R_1$. If $r=1$ and $n\ge 2$, then
	\[
		\pi^2X^{N-b}=\pi^2X^{(n-2)(b-1)+b-1}\in (\pi^n X^{b-1},\pi^{n+1})R_1.
	\]
\end{proof}

\begin{thm} \label{thm one dimensional}
	Suppose that $b\ge 3$ and $\bar{G}\in x^2 k\llbracket x \rrbracket$. Then $R$ does not admit a preperfectoid tower arising from $(R,(p))$.
\end{thm}
\begin{proof}
	There exists $c \in R_1$ such that $t_0(x)=X^p+pc$.  Let $S$ be the completion of $R_1$ with respect to $(\pi, X)$.
	Hence,
	\begin{align}\label{one dimensional expansion eq}
	\begin{aligned}
		0 &= t_0(x^b+pG)\\
		&= (X^p+pc)^b+pG(X^p+pc)\\
		& = puh^p+\sum_{i=1}^b \binom{b}{i}X^{p(b-i)}p^ic^i+pG(X^p+pc) \in S.
	\end{aligned}
	\end{align}
	Since $S$ is $p$-torsionfree by Lemma \ref{lem completion of R1} and $p(b-1)\ge b$, we obtain
	\[
		F_0(\bar{u})h_0^p+\bar{G}^p=0 \in A.
	\]
	Since $\bar{G}$ is not a unit, $h_0\in xA$. Hence, by Lemma \ref{first reduction lemma},
	$\mathrm{ord}_x h_0=1$.

	If $b>p$, then $\mathrm{ord}_x(h_0^p)=p$. Hence, $\mathrm{ord}_x (\bar{G}^p)=p$. This is a contradiction since $\bar{G}\in x^2A$ by the assumption. Hence, we may assume that $b\le p$.

Take $n\in \Z_{\ge0}$ and $1\le r\le b-1$ such that $p=n(b-1)+r$.
Then by Lemma \ref{second reduction lemma} and (\ref{h formula}), we see that $\mathrm{ord}_\pi(X^p)=\mathrm{ord}_{\pi}(h^p)=n$.
Applying Lemma \ref{second reduction lemma} for $N=p(b-1)=(p-1)(b-1)+(b-1)$, we see that $\mathrm{ord}_\pi(X^{p(b-1)})=p-1$.
There exist $H,K\in W(k)\llbracket x \rrbracket$ such that $G=x^2H+pK$. Hence,
\[
	G(X^p+pc)\in (\pi^{2n},p)S\subseteq \pi^{n+1} S.
\]
By (\ref{one dimensional expansion eq}), since $p-1>n$,
\[
	0=uh^p+\sum_{i=1}^b \binom{b}{i}X^{p(b-i)}p^{i-1}c^i+G(X^p+pc) \in \pi^{n}S\setminus \pi^{n+1}S.
\]
This is a contradiction.
\end{proof}

\begin{cor} \label{cor $p^a+x^b$}
	Let $a,b\in \Z_{>0}$ and $R:=W(k)\llbracket x \rrbracket/(p^a+x^b)$. Then
	$R$ admits a perfectoid tower arising from $(R,(p))$ if and only if one of the following conditions holds.
	\begin{enumerate}
		\item $a=1$ or $b=1$.
		\item $p\ge 3$ and $a=b=2$.
	\end{enumerate}
\end{cor}
\begin{proof}
	If $a=1$ or $b=1$, then $R$ admits a perfectoid tower arising from $(R,(p))$ by \cite[Example 3.62 (1)]{INS25}, since $R$ is a regular local ring.

If $p\ge 3$ and $a=b=2$, let $R_i:=W(k)\llbracket s_i, x_i\rrbracket/(s_i^{p^i}-p,x_i^2+s_i^2)$ and let
\[
	t_i:R_i\to R_{i+1};\  s_i\mapsto s_{i+1}^{p},\  x_i\mapsto x_{i+1}^p.
\]
Then $(\{R_i\}_{i\ge 0},\{t_i\}_{i\ge 0})$ is a perfectoid tower arising from $(R,(p))$.

By Theorem \ref{thm one dimensional}, $R$ does not admit a preperfectoid tower if $b\ge 3$ and $a\ge 2$.

Suppose that $p=2$,  $a\ge 2$ and $b=2$. We show that $R$ does not admit a preperfectoid tower arising from $(R,(p))$. By (\ref{one dimensional expansion eq}), we see that
\[
	uh^2+2X^2c+2c^2+2^{a-1}=0.
\]
Hence, $h^2\in \pi^2 R_1$.
Since $h_0^2=0$ and $A\cong k[x]/(x^2)$, there exists $\lambda\in k$ such that $h_0=\lambda x$.
If $\lambda=0$, then $X^2=\pi h\in \pi^2 R_1=2R_1$. Hence,
\[
\bar{t_0}(x)=\bar{X}^2= 0 \in R_1/2R_1,
\]
which contradicts the injectivity of $\bar{t_0}$. Therefore, $\lambda\neq 0$. Let $\lambda_1\in R_1$ be a lift of $\lambda\in k\subseteq  A$. Then
\[
	h\equiv \lambda_1 X \pmod{\pi R_1}.
\]
Since
\[
	h^2\equiv \lambda_1^2X^2=\lambda_1^2\pi h \pmod{\pi^2R_1},
\]
we have $\pi h\in \pi^2 R_1$. Since $R_1$ is $\pi$-torsionfree, $h\in \pi R_1$. Hence, $h_0=0$, which is a contradiction.

It remains to consider $p\ge3$, $a\ge3$, and $b=2$.

By (\ref{one dimensional expansion eq}), $h_0^p=0\in A\cong k[x]/(x^2)$. Hence, there exists $\lambda\in k$ such that $h_0=\lambda x$. Let $\lambda_1\in R_1$ be a lift of $\lambda \in k\subseteq A$. Replacing $X$ by $X-2^{-1}\pi\lambda_1$, we may assume that $h_0=0$. Then there exists $h'\in R_1$ such that $h=\pi h'$ and $X^2=\pi^2h'$.
We have
\[
	X^p=X\pi^{p-1}h'^{\frac{p-1}{2}}.
\]
Let $h_0'\in A$ be the image of $h'$. Then $h'_0$ is a unit. Indeed, if not, $Xh'^{(p-1)/2}\in \pi R_1$ and $X^p\in pR_1$. This is a contradiction since $\bar{t_0(x)}=\bar{X}^p\neq 0$.
By (\ref{one dimensional expansion eq}), we obtain
\[
	uh^p+2X^{p}c+pc^2+p^{a-1}=0.
\]
Hence,
\[
	2Xh'^{\frac{p-1}{2}}c+\pi (uh'^p+u^{-1}c^2)\equiv 0 \pmod{\pi^2 R_1}.
\]
Reducing modulo $\pi$, we see that $xc_0=0\in A$, where $c_0\in A$ is the image of $c\in R_1$. Hence, $c_0\in xA$. Therefore, there exist $c_1, c_2\in R_1$ such that $c=Xc_1+\pi c_2$. Then we have
\[
	2x{h'}_0^{\frac{p-1}{2}}c_{2,0}+u_0{h'}_0^p=0 \in A,
\]
where $c_{2,0}$ and $u_0$ are the images of $c_2$ and $u$ in $A$, respectively. This is a contradiction since the first term is not a unit, but the second term is a unit.

\end{proof}

\section{A lift of a rational double point}
\label{section $p^2+x^p+yz$}
We give a mixed-characteristic lift $R$ of an $F$-liftable rational double point that admits no preperfectoid tower arising from $(R,(p))$.
\begin{thm} \label{thm $p^2+x^p+yz$}
Let $k$ be a perfect field of characteristic $p>0$, and let
\[
  R:=W(k)\llbracket x,y,z\rrbracket/(p^2+x^p+yz).
\]
Then $R$ does not admit a preperfectoid tower arising from $(R,(p))$.
\end{thm}
\begin{proof}
Suppose that $R$ admits a preperfectoid tower arising from $(R,(p))$. Choose $X,Y,Z\in R_1$ such that $F_0(\bar X)=\bar x$, $F_0(\bar Y)=\bar y$, and $F_0(\bar Z)=\bar z$ in $A$, and set $F=X^p+YZ$. Choose a unit $u\in R_1^\times$ such that $\pi^p=pu$. There exist elements $a,b,c\in R_1$ such that
\[
	t_0(x)=X^p+pa, \quad t_0(y)=Y^p+pb, \quad t_0(z)=Z^p+pc.
\]
Let $(S,\n)$ be the completion of $R_1$ with respect to $(\pi,X,Y,Z)$. Then $S$ is a $\pi$-torsionfree Noetherian local ring by Lemma \ref{lem completion of R1}.

\begin{cl}
	Changing the choice of lifts $X,Y,Z$, we may assume that $F\in \pi^2 R_1$ and $b,c\in (\pi,X,Y,Z)$.
\end{cl}
\begin{clproof}
	Since $F_0(\bar{F})=0$, $F=\pi h_1$ for some $h_1\in R_1$.
	Let
	\[
		\Delta:=\frac{X^{p^2}+Y^pZ^p-F^p}{p}.
	\]
	Then
	\begin{align} \label{X^p+YZ expansion}
		\begin{aligned}
		0&=t_0(p^2+x^p+yz)\\
		 &\equiv F^p+p\Delta+p(Y^pc+Z^pb) \pmod{p^2 R_1}.
		\end{aligned}
	\end{align}
	If $p\ge 3$, then $F$ divides $\Delta$. If $p=2$, then $\Delta=-X^2YZ$. Hence, $F_0(\bar{\Delta})=-y^2z^2\in R/pR$ if $p=2$. Dividing (\ref{X^p+YZ expansion}) by $p$ and reducing modulo $\pi$, we obtain
\[
		F_0(\bar{h_1})^p\in (y^p,z^p)R/pR.
\]
Hence, $F_0(\bar{h_1})\in (y,z)R/pR$. Therefore, there exist $g_1,g_2,g_3\in R_1$ such that
\[
	h_1=Yg_1+Zg_2+\pi g_3.
\]
Replacing $Y,Z$ by $Y-\pi g_2$, $Z-\pi g_1$ gives $X^p+YZ\in \pi^2 R_1$.
Dividing (\ref{X^p+YZ expansion}) by $p$, and reducing modulo $\pi$, we have
\[
	y^pF_0(\bar{c})+z^pF_0(\bar{b})+\epsilon=0 \in R/pR, \quad \epsilon =\begin{cases}
	0, & (p\ge 3),\\
	y^2z^2, & (p=2).
	\end{cases}
\]
Hence, $b,c \in (\pi,X,Y,Z)R_1$.
\end{clproof}

For these adjusted lifts, write $F=X^p+YZ=\pi^2h$ with $h\in R_1$.

Considering a surjection
\[
W(k)\llbracket P,X,Y,Z\rrbracket \twoheadrightarrow S; P\mapsto \pi, X\mapsto X, Y\mapsto Y, Z\mapsto Z,
\]
choose lifts $U,H\in W(k)\llbracket P,X,Y,Z\rrbracket$ of $u,h$, respectively, so that
\[
	S \cong W(k)\llbracket P,X,Y,Z \rrbracket/(P^p-pU, X^p+YZ-P^2H).
\]

\noindent\textbf{Case 1: $p\ge 3$.}
	We first prove that $h$ is a unit. Otherwise, $H\in (P,X,Y,Z)$. Since
\[
	\gr_{\n}S\cong k[P,X,Y,Z]/(YZ),
\]
we have $p^2=u^{-2}\pi^{2p}\notin \n^{2p+1}$. Since $YZ=-X^p+\pi^2h \in \n^3$ and $b,c\in \n$,
\[
	t_0(x)^p\in \n^{p^2}\subseteq \n^{2p+1}, \quad t_0(yz)=(YZ)^p+p(Y^pc+Z^pb)+p^2bc \in \n^{2p+1},
\]
which contradicts $p^2+t_0(x)^p+t_0(yz)=0$. Therefore, $h$ is a unit.

We have
\[
	\Delta=\sum_{i=1}^{p-1}\frac{1}{p}\binom{p}{i}F^i(-X^p)^{p-i}\equiv \pi^2hX^{p(p-1)} \pmod{\pi^3 R_1}.
\]
Let $S':=S/(Y^p,Z^p)$. Note that $S'$ is $\pi$-torsionfree.
Passing to $S^{\prime}$, dividing (\ref{X^p+YZ expansion}) by $\pi^2p$, and reducing modulo $\pi$, we obtain
\[
	hX^{p(p-1)}=0 \in S'/\pi S'\cong R/(p,y^p,z^p).
\]
Since $h$ is a unit, we obtain
\[
	x^{p(p-1)}=(yz)^{p-1}=0 \in R/(p,y^p,z^p)R\cong k\llbracket x,y,z\rrbracket/(x^p+yz,y^p,z^p).
\]
This is a contradiction.

\noindent\textbf{Case 2: $p=2$.}
Since $X^2+YZ=\pi^2 h$, $\pi^2=2u$ and $t_0(4+x^2+yz)=0$, we obtain
\[
	2D+Y^2c+Z^2b+Y^2Z^2=0,
\]
where $D:=1+u^2h^2-uhYZ+X^2a+a^2+bc$.
The images of $Y^2,Z^2$ form a regular sequence on $R_1/\pi R_1$, corresponding to $y^2,z^2$ under $R_1/\pi R_1\cong R/2R\cong k\llbracket x,y,z\rrbracket/(x^2+yz)$. Hence, $b\in (\pi,Y^2)R_1$ and $c\in (\pi,Z^2)R_1$.
Since $S/(Y^2,Z^2)$ is $\pi$-torsionfree, $D\in (Y^2,Z^2)S$.
Reducing modulo $\pi$ and setting $g:=F_0(\bar{uh+a})$, we obtain
\[
	1+g^2+gyz=0\in T:=k[x,y,z]/(x^2+yz,y^2,z^2).
\]
Write $g=\alpha+\beta x+\gamma$, where $\alpha,\beta\in k$ and $\gamma\in (y,z)T$. Then
\[
	0=1+\alpha^2+(\beta^2+\alpha)yz+\beta xyz\in T.
\]
Comparing the constant and $xyz$ coefficients gives $\alpha=1$ and $\beta=0$, contradicting the $yz$ coefficient.
\end{proof}

\appendix
\section{Proof of Lemma \ref{lem evaluation g}} \label{appendix homogeneous component}
We use the notation of Subsection \ref{subsection coefficient obstruction}.
\subsection{Associated graded rings and an order estimate}
Take a lift $\tilde{u}$ of $u$ and let
\[
	T:=W(k)\llbracket P,X,Y,L \rrbracket/(P^p-\tilde{u}p).
\]
Then $T$ is a regular local ring and we have a natural surjection $T\twoheadrightarrow S$, $P\mapsto\pi$. Let $I$ be its kernel. Since $S$ is $\pi$-torsionfree, $I\cap PT=PI$. Thus
\[
 I/PI\cong\ker(T/(P)\to S/\pi S)=(L^p).
\]
Choosing $F\in I$ lifting $L^p$, Nakayama's lemma gives
\[
 I=(F),\qquad S\cong T/(F),\qquad F\equiv L^p\pmod{PT}.
\]

Let $\n:=(P,X,Y,L)$ be the maximal ideal of $T$. We consider the associated graded ring $\gr_\n(T)$. Note that
\[
	\gr_\n(T)\cong k[P,X,Y,L].
\]

Since $p=P^{p}\tilde{u}^{-1}$, $\mathrm{ord}_\n(p)=p$.
Fix lifts $\tilde{a}, \tilde{b}, \tilde{c}\in T$ of $a,b$ and $c$.
Since $k$ is perfect, there exist $\alpha, \beta, \gamma\in k$ such that
\begin{align*}
	\mathrm{in}_{\n}(X^p+p\tilde{a}) &= (X+\alpha P)^p, \\
	\mathrm{in}_{\n}(Y^p+p\tilde{b}) &= (Y+\beta P)^p, \\
	\mathrm{in}_{\n}(Z^p+p\tilde{c}) &= (L-X-Y+\gamma P)^p.
\end{align*}

Let $\lambda:=\alpha+\beta+\gamma$. Since $t_0(x^p+y^p+z^p)=0 \in S$, we have
\[
	Q:=(X^p+p\tilde{a})^p+(Y^p+p\tilde{b})^p+(Z^p+p\tilde{c})^p\in (F)
\]
and
\[
	\mathrm{in}_\n(Q)=(L+\lambda P)^{p^2}.
\]
Let $\tilde{\lambda}\in T$ be a lift of $\lambda$. Replacing $Z$ by $Z+\tilde{\lambda}P$, we may assume that
\[
	\mathrm{in}_\n(Q)=L^{p^2}.
\]
We redefine $L,h,a,b,c,D,G$ and $g$ using these adjusted lifts. The identities in Section \ref{section second obstruction} continue to hold in $S$.

We also use $\Delta$ for the polynomial $\Delta(X,Y,Z)$ in $T$. Then we see that
\[
	L^p=\pi h -p\Delta \text{ in $S$}.
\]
Let $\eta:=\tilde{h}-P^{p-1}\tilde{u}^{-1}\Delta$, where $\tilde{h}, \tilde{u}\in T$ are lifts of $h$ and $u$, respectively. Then
\[
L^p \equiv P\eta \pmod{(F)T}.
\]
Note that  $\bar{\eta}=\bar{h}\in S/\pi S\cong A$.
\begin{lem} \label{tangent cone lemma}
	With notation as above, $\mathrm{in}_\n(F)=L^p$ and $h\in \m_S^p$.
\end{lem}
\begin{proof}
	Since $Q\in (F)$ and $\gr_\n T\cong k[P,X,Y,L]$ is a domain, $\mathrm{in}_\n(F)|L^{p^2}$. Hence,
	\[
		\mathrm{in}_\n(F)=\mu L^r
	\]
	for some $\mu \in k^\times$, where $r:=\mathrm{ord}_\n(F)$. Since $F\equiv L^p \pmod P$, $r=p$ and $\mu=1$.
	Since $L^p-P\eta\in (F)T$ and $L^p\equiv F \pmod{(P)T}$, $\mathrm{in}_\n(L^p-P\eta)=\mathrm{in}_\n(F)=L^p$. Hence, $\eta\in \n^p$. Since $\Delta\in \m_S^p$, we obtain $h\in \m_S^p$.
\end{proof}

Reducing (\ref{second ob expansion D}) modulo $\pi$, we obtain
\[
	\bar{u}\bar{h}^p+\bar{\Delta}^p=0 \in S/\pi S.
\]
Therefore,
\[
	\rho(\bar{\Delta})^p=-\rho(u)\rho(\bar{h})^p.
\]
Since $C$ is normal, there exists $v\in C^\times$ such that
\[
	\rho(\bar{\Delta})=v\rho(\bar{h}).
\]
We take a lift $v_1\in S^\times$ of $v$.

\begin{cor} \label{cor carry}
	Let $S':=S/(h^p)$ and $\m_{S'}$ be its maximal ideal. Then $P$ is a non-zero-divisor in $\gr_{\m_{S'}}S'$. Consequently, for any $N, r\ge1$, we have
\[
	P^rS'\cap \m_{S'}^{N+r}=P^r\m_{S'}^{N}.
\]
\end{cor}
\begin{proof}
	By Lemma \ref{tangent cone lemma}, $\gr_{\m_S}{S}\cong k[P,X,Y,L]/(L^p)$.
	Since $\mathrm{in}_{\m_{S}}(h)$ is a non-zero-divisor on  $\gr_{\m_S}{S}$,
	\[
		\gr_{\m_{S'}}{S'}\cong \frac{k[P,X,Y,L]}{(L^p,\mathrm{in}_{\m_{S}}(h)^p)}.
	\]
	Hence, $P$ is a non-zero-divisor on this ring. Thus, for any $r\ge 1$ and $f\in S'\setminus\{0\}$,
	\[
		\mathrm{ord}_{\m_{S'}}(P^rf)=r+\mathrm{ord}_{\m_{S'}}(f),
	\]
	which gives the desired equality.
\end{proof}

\subsection{Evaluation of $\epsilon$}
 Let $\epsilon:=\Delta-v_1h$ and $U:=u+v_1^p$.
Then $\rho(\bar{\epsilon})=0$. By (\ref{second ob expansion D}),
we see that
\[
	Uh^p+\epsilon^p \equiv uh^p+\Delta^p\equiv 0 \pmod{pS}.
\]
\begin{lem} \label{lem evaluation epsilon}
	With notation as above,
	\[
		F_0(\bar{\epsilon})\in(x,y,l)^{p+1}A.
	\]
\end{lem}
\begin{proof}
	Since $\rho(\bar{\epsilon})=0$, there exist $g_i\in k\llbracket x,y\rrbracket$ such that
	\[
		F_0(\bar{\epsilon}) = \sum_{i=1}^{p-1}l^ig_i.
	\]
	Let $e_i:=\mathrm{ord}_{(x,y)}g_i$. Since $\epsilon\in \m_S^p$, it follows that $i+e_i\ge p$.

	For each $i$, choose a lift $f_i\in \m_S^{e_i}$ of $g_i$. If $g_i=0$, then take $f_i=0$. Take $f_0\in S$ such that
	\[
		\epsilon=\sum_{i=1}^{p-1}L^if_i+\pi f_0.
	\]
	Since $\eta \equiv h \pmod {\pi^{p-1}S}$ and $Uh^p+\epsilon^p\in pS$, we obtain
	\[
		Uh^p+\sum_{i=1}^{p-1}\pi^ih^if_i^p \equiv 0 \pmod{pS}.
	\]
	Let $S':=S/(h^p)$ and $A_i:=h^if_i^p$. Then
	\[
		\sum_{i=1}^{p-1}P^iA_i \in pS'.
	\]

	Suppose that $i_0$ is the smallest index $i$ such that $i+e_i=p$. For $i<i_0$, $\mathrm{ord}_{\m_{S}} A_i\ge p^2+p$.
	Set $B_i:=\sum_{j=1}^{i-1}P^{j-1}A_j$. Then
	\begin{align} \label{A B eq}
		B_i\in P^{i-1}S', \ \frac{B_i}{P^{i-1}}+A_i\in PS'.
	\end{align}
	By Corollary \ref{cor carry}, $\mathrm{ord}_{\m_{S'}}(B_{i_0}/P^{{i_0}-1})\ge p^2+p-({i_0}-1)>p^2$. Consider the quotient
	\[
		S'\to S'/(P,L)\cong C/(H^p),
	\]
	where $H:=\rho(\bar{h})\in C\cong k\llbracket x, y\rrbracket$. Note that $\mathrm{ord}_{(x,y)}H=p$ and the $p$-th homogeneous part $[H]_p$ of $H$ is
	\[
		v'\Delta(x,y,-x-y)
	\]
	for some $v'\in k^\times$.
	Hence, $\gr_{(x,y)}(C/(H^p))\cong k[x,y]/([H]_p^p)$.
	By (\ref{A B eq}), it follows that
	\[
		[H]_p^{i_0}[g_{i_0}]_{p-i_0}^p=0 \in k[x,y]/([H]_p^p).
	\]
	Hence, $[H]_p^{p-i_0}\mid [g_{i_0}]_{p-i_0}^p$. Since $\deg ([H]_p^{p-i_0})=\deg [g_{i_0}]_{p-i_0}^p=p(p-i_0)$, there exists $v'_1\in k^\times$ such that $[H]_p^{p-i_0}=v'_1[g_{i_0}]_{p-i_0}^p$. Hence,
\[
	p \, \mathrm{ord}_{x}([g_{i_0}]_{p-i_0})=p-i_0,
\]
which is a contradiction. Hence, $i+e_i\ge p+1$ for any $1\le i\le p-1$.
\end{proof}

\subsection{Completion of the proof}
\begin{proof}[Proof of Lemma \ref{lem evaluation g}]
	By Lemma \ref{lem evaluation epsilon} and $F_0(\bar{\epsilon})\in lA$, there exist $r\in \m_S^p$ and $s\in S$ such that
	\[
		\epsilon=Lr+\pi s.
	\]
	Then
	\begin{align*}
		pG&=uh^p+\Delta^p=uh^p+(v_1h+\epsilon)^p \\
		& =Uh^p+L^pr^p+\pi^ps^p +p B,
	\end{align*}
	for some $B\in (\pi,L)S$.

	Let $V:=Uh^p+L^pr^p$. Then $V\in pS$ and
	\[
		L^pr^p=\pi \eta r^p\in \m_{S}^{p^2+p+1}.
	\]
	By Corollary \ref{cor carry}, it follows that
	\[
		\frac{V}{\pi^p}\in (h^p,\m_{S}^{p^2+1}).
	\]
	Hence, there exists $v'_2\in k$ such that
	\[
		\left[\rho\left(\frac{V}{\pi^p}\right)\right]_{p^2}=v_2'[H]_p^p,
	\]
	where $H:=\rho(\bar{h})$. Since $\rho(u)=-v^p$,
	\[
		g=-v^p\rho\left(\frac{V}{\pi^p}\right)-(v\rho(s))^p.
	\]
	Taking the $p^2$-th homogeneous components of both sides, the conclusion follows.
\end{proof}

\bibliography{mixed.bib}
\bibliographystyle{alpha}

\bigskip

\end{document}